\documentclass[11pt]{article}
\usepackage{geometry}                
\usepackage{graphicx}
\usepackage{enumerate}
\usepackage{setspace}
\usepackage{amssymb}
\usepackage{amsmath, amsthm}
\usepackage{amsmath}
\usepackage{amsthm}
\usepackage{blkarray}
\usepackage{epstopdf}
\usepackage{footnote}
\newtheorem{thm}{Theorem}[section]

\newtheorem{definition}{Definition}[section]

\newtheorem{lemma}{Lemma}[section]
\newtheorem{corollary}{Corollary}[section]
\begin{document}\large{
\title{Direct and inverse results
on row sequences of simultaneous Pad\'{e}-Faber approximants}
\author{N. Bosuwan\thanks{The research of N. Bosuwan was supported by the Strengthen Research Grant for New Lecturer from the Thailand Research Fund and the Office of the Higher Education Commission (MRG6080133) and Faculty of Science, Mahidol University.}\,\,\,\footnote{Corresponding author.} and  G. L\'opez Lagomasino\thanks{The research of G. L\'opez Lagomasino was supported by research grant MTM2015-65888-C4-2-P from Ministerio de Econom\'ia, Industria y Competitividad}}}
\maketitle

\begin{abstract} Given a vector function $\textup{\textbf{F}}=(F_1,\ldots,F_d),$ analytic on a neighborhood of some compact subset $E$ of the complex plane with simply connected complement, we define a sequence of vector rational functions with common denominator in terms of the expansions of the components $F_k, k=1,\ldots,d,$ with respect to the sequence of Faber polynomials associated with $E$. Such sequences of vector rational functions are analogous to row sequences of type II Hermite-Pad\'e approximation. We give necessary and sufficient conditions for the convergence with geometric rate of the common denominators of the sequence of vector rational functions so constructed. The exact rate of convergence of these denominators is provided and the rate of convergence of the approximants is estimated. It is shown that the common denominators of the approximants detect the poles of the system of functions ``closest'' to $E$ and their order.
\end{abstract}


\vspace{0,7cm}


\vspace{0,5cm}\noindent

\section{Introduction}

The object of this paper is to prove a Montessus de Ballore-Gonchar type theorem for simultaneous Pad\'{e}-Faber approximants analogous to the one obtained in \cite{CacoqYsernLopez} in the context of Hermite-Pad\'e approximation. Such results, motivated in \cite{Aag81}, include a direct part where convergence of the approximants and their poles is derived provided that the functions being approximated have convenient analytic properties, and an inverse statement in which starting out from the asymptotic properties of the poles of the approximants some important analytic properties of the functions being approximated are determined. For scalar functions, several approximating models have been explored which in one way or another extend the notion of Pad\'e approximation, for example, see \cite{Vib2009,Aag81,Sutinpade}. To avoid unnecessary repetitions, in the introduction of   \cite{Vib2009,CacoqYsernLopez,MorrisSaff} you can find an account of the history of the problem. We wish to mention that in \cite{BosuwanLopez} we studied a similar problem when the approximants are built on the basis of orthogonal expansions.

\medskip

Let us clarify what we understand as a pole of a vector function and its order.

\begin{definition}\textup{
Let $\boldsymbol \Omega:=(\Omega_1,\Omega_2,\ldots,\Omega_d)$ be a system of domains such that, for each $\alpha=1,2,\ldots,d,$ $F_{\alpha}$ is meromorphic in $\Omega_{\alpha}.$ We say that the point \emph{$\lambda$ is a pole of $\textup{\textbf{F}}:=(F_1,F_2,\ldots,F_d)$ in $\boldsymbol \Omega$ of order $\tau$} if there exists an index $\alpha\in \{1,2,\ldots,d\}$ such that $\lambda \in \Omega_{\alpha}$ and it is a pole of $F_{\alpha}$ of order $\tau,$ and for $\beta\not=\alpha$ either $\lambda$ is a pole of $F_{\beta}$ of order less than or equal to $\tau$ or $\lambda\not\in \Omega_{\beta}.$ When $\boldsymbol \Omega=(\Omega,\Omega,\ldots,\Omega),$ we say that \emph{$\lambda$ is a pole of $\textup{\textbf{F}}$ in $\Omega.$}
}
\end{definition}

Let $E$ be a compact subset of the complex plane $\mathbb{C}$ such that  $\overline{\mathbb{C}}\setminus E$ is simply connected and $E$ contains more than one point. It is convenient to assume that $0\in E$ and this can be done, if necessary, without loss of generality making a change of variables. There exists a unique exterior conformal mapping $\Phi$ from $\overline{\mathbb{C}}\setminus E$ onto $\overline{\mathbb{C}}\setminus \{w\in \mathbb{C}: |w|\leq 1\}$ satisfying $\Phi(\infty)=\infty$ and $\Phi'(\infty) := \lim_{z\to \infty} \Phi(z)/z >0.$ It is well known that $\Phi'(\infty)= 1/\mbox{\rm cap}(E)$ where $\mbox{\rm cap}(E)$ is the logarithmic capacity of $E$.
For any $\rho>1,$  we define
$$\Gamma_{\rho}:=\{z\in \mathbb{C}: |\Phi(z)|=\rho\} \quad \quad \mbox{and} \quad \quad D_{\rho}:=E\cup \{z\in \mathbb{C}: |\Phi(z)|<\rho\},$$
as the \emph{level curve of index $\rho$} and the \emph{canonical domain of index $\rho$}, respectively.

Denote by $\mathcal{H}(E)$ the space of all functions holomorphic in some neighborhood of $E.$ We define
$$\mathcal{H}(E)^d:=\{(F_1,F_2,\ldots,F_d): \textup{$F_\alpha\in \mathcal{H}(E)$ for all $\alpha=1,2,\ldots,d$}\}.$$

\medskip

Let $\textup{\textbf{F}}\in \mathcal{H}(E)^d.$ Denote by $\rho_0(\textup{\textbf{F}})$ the index $\rho$ of the largest canonical domain $D_{\rho}$ to which all $F_{\alpha},$ $\alpha=1,\ldots,d,$ can be extended as holomorphic functions and by $\rho_{m}(\textup{\textbf{F}})$  the index $\rho$ of the largest canonical domain $D_\rho$ to which all $F_\alpha,$ $\alpha=1,\ldots,d$ can be extended so that $\textup{\textbf{F}}$ has at most $m$ poles counting multiplicities.

\medskip

The \emph{Faber polynomial} of $E$ of degree $n$ is defined by the formula
\begin{equation}\label{polynomial}
 \Phi_n(z):=\frac{1}{2\pi i}\int_{\Gamma_{\rho}} \frac{\Phi^n(t)}{t-z}dt, \quad\quad z\in D_{\rho}, \quad \quad n=0,1,2,\ldots.
 \end{equation}
It equals the polynomial part of the Laurent expansion of $\Phi^n$ at infinity. Notice that
\begin{equation}
\label{capa}
\Phi_n(z) =  \left(z/\mbox{\rm cap}(E)\right)^n + \mbox{lower degree terms}.
\end{equation}
The $n$-th \emph{Faber coefficient} of $G\in \mathcal{H}(E)$ with respect to $\Phi_n$ is given by
\begin{equation*}\label{Fourierco}
[G]_n:=\frac{1}{2\pi i}\int_{\Gamma_\rho} \frac{G(t) \Phi'(t)}{\Phi^{n+1}(t)} dt,
\end{equation*}
where $\rho\in (1,\rho_{0}(G))$ and $\rho_{0}(G)$ denotes the index of the largest canonical region to which $G$ can be extended as a holomorphic function.
For an account on Faber polynomials and its properties see \cite{SmirnovLebedev,Suetin}. In particular, it is well known that
\begin{equation}
\label{conv}
\lim_{n\to \infty} |\Phi_n(z)|^{1/n} = |\Phi(z)|,
\end{equation}
uniformly on compact subsets of $ \mathbb{C} \setminus E$.

\medskip

Let us introduce simultaneous Pad\'{e}-Faber approximants.

\begin{definition}\label{simu}\textup{ Let $\textup{\textbf{F}}=(F_1,\ldots,F_d)\in \mathcal{H}(E)^d.$  Fix  $\textup{\textbf{m}}=(m_1,\ldots,m_d)\in \mathbb{N}^d$ and  $n\in \mathbb{N}.$  Set $|\textup{\textbf{m}}|:=m_1+m_2+\ldots+m_d.$ Then, there exist polynomials $Q_{n,\textup{\textbf{m}}},$ $P_{n,\textup{\textbf{m}},k,\alpha},$   such that
\begin{equation}\label{simu1}
\deg P_{n,\textup{\textbf{m}},k,\alpha}\leq n-1, \quad \quad  \deg(Q_{n,\textup{\textbf{m}}})\leq |\textup{\textbf{m}}|, \quad \quad Q_{n,\textup{\textbf{m}}}\not\equiv 0,
\end{equation}
\begin{equation}\label{simu2}
[ Q_{n,\textup{\textbf{m}}} z^{k} F_{\alpha}-P_{n,\textup{\textbf{m}},k,\alpha}]_j=0,   \quad \quad j=0,1,\ldots,n,
\end{equation}
for all $k=0,1,\ldots,m_{\alpha}-1$ and $\alpha=1,2,\ldots,d.$
The vector of rational functions
$$\textup{\textbf{R}}_{n,\textup{\textbf{m}}}:=(R_{n,\textup{\textbf{m}},1},\ldots,R_{n,\textup{\textbf{m}},d})=(P_{n,\textup{\textbf{m},0,1}}, \ldots, P_{n,\textup{\textbf{m},0,d}})/Q_{n,\textup{\textbf{m}}}$$ is called an \emph{$(n,\textup{\textbf{m}})$ simultaneous Pad\'{e}-Faber approximant of $\textup{\textbf{F}}.$}}
\end{definition}

Clearly,
\begin{equation}\label{denosim}
[Q_{n,\textup{\textbf{m}}} z^{k} F_{\alpha}]_n=0,\quad \quad \alpha=1,\ldots,d,  \quad \quad k=0,1,\ldots,m_{\alpha}-1.
\end{equation} Since $Q_{n,\textup{\textbf{m}}}\not\equiv 0,$ we normalize it to have leading coefficient equal to $1.$ We call $Q_{n,\textup{\textbf{m}}}$ \emph{the denominator of the $(n,\textup{\textbf{m}})$ simultaneous Pad\'{e}-Faber approximant of $\textup{\textbf{F}}$}.

\medskip

Finding a solution of \eqref{simu1}-\eqref{simu2} reduces to solving a homogeneous system of $(n+1)|\textup{\textbf{m}}|$ linear equations on $(n+1)|\textup{\textbf{m}}|+1$  coefficients of $Q_{n,\textup{\textbf{m}}}$ and $P_{n,\textup{\textbf{m}},k,\alpha}.$ Therefore, for any pair $(n,\textup{\textbf{m}})\in \mathbb{N}\times \mathbb{N}^d,$ a vector of rational functions $\textup{\textbf{R}}_{n,\textup{\textbf{m}}}$ always exists. In general, it may not be unique. For each $n,$ we choose one solution.
The definition of simultaneous Pad\'{e}-Faber approximants employed here differs from the one used in \cite{Bosuwan} which may seem more natural but has serious inconveniences for proving inverse type results.

\medskip

Notice that \eqref{simu2} implies that linear combinations of the functions $z^kF_\alpha, 0\leq k < m_\alpha, \alpha = 1,\ldots,d$ also verify \eqref{simu2} (with the same $Q_{n,{\bf m}}$ and convenient polynomial $P,$ $ \deg P < n$). This motivates the concept of system pole. Systems poles may not coincide with the poles of the individual functions $F_\alpha$ (see examples in \cite{CacoqYsernLopez}).

\begin{definition}\label{systempoleofordernew}
\textup{ Given $\textup{\textbf{F}}=(F_1,\ldots,F_d)\in\mathcal{H}(E)^d$ and $\textup{\textbf{m}}=(m_1,\ldots,m_d)\in \mathbb{N}^d $, we say that $\xi\in \mathbb{C}$ is a \emph{system pole of order $\tau$ of $\textup{\textbf{F}}$ with respect to $\textup{\textbf{m}}$} if $\tau$ is the largest positive integer such that for each $t=1,2,\ldots,\tau,$ there exists at least one polynomial combination of the form
\begin{equation}\label{polycom}
\sum_{\alpha=1}^d v_{\alpha} F_{\alpha},\quad \quad \deg (v_{\alpha})<m_{\alpha},\quad \quad \alpha=1,2,\ldots,d,
\end{equation}
which is holomorphic on a neighborhood of $\overline{D}_{|\Phi(\xi)|}$ except for a pole at $z=\xi$ of exact order $t.$ }
\end{definition}

To each system pole $\xi$ of $\textup{\textbf{F}}$ with respect to $\textup{\textbf{m}},$ we associate several characteristic values. Let $\tau$ be the order of $\xi$ as a system pole of $\textup{\textbf{F}}.$ For each $t=1,\ldots,\tau,$ denote by $\rho_{\xi,t}(\textup{\textbf{F}},\textup{\textbf{m}})$ the largest of all the numbers $\rho_{t}(G)$ (the index of the largest canonical domain containing at most $t$ poles of $G$), where $G$ is a polynomial combination of type \eqref{polycom} that is holomorphic on a neighborhood of $\overline{D}_{|\Phi(\xi)|}$ except for a pole at $z=\xi$ of order $t.$ There is only a finite number of such possible values so the maximum is indeed attained. Then, we define
$$\boldsymbol\rho_{\xi,t} (\textup{\textbf{F}}, \textup{\textbf{m}}):=\min_{k=1,\ldots,t} \rho_{\xi, k}(\textup{\textbf{F}}, \textup{\textbf{m}}),$$
$$\boldsymbol\rho_{\xi} (\textup{\textbf{F}}, \textup{\textbf{m}}):=\boldsymbol\rho_{\xi,\tau} (\textup{\textbf{F}}, \textup{\textbf{m}})=\min_{k=1,\ldots, \tau} \rho_{\xi, k} (\textup{\textbf{F}},\textup{\textbf{m}}).$$

Fix $\alpha\in\{1,\ldots,d\}$. Let $D_{{\alpha}}(\textup{\textbf{F}},\textup{\textbf{m}})$ be the largest canonical domain in which all the poles of $F_{\alpha}$ are system poles of $\textup{\textbf{F}}$ with respect to $\textup{\textbf{m}},$ their order as poles of $F_{\alpha}$ does not exceed their order as system poles, and $F_{\alpha}$ has no other singularity. By $\boldsymbol\rho_{\alpha}(\textup{\textbf{F}},\textup{\textbf{m}}),$ we denote the index of this canonical domain. Let $\xi_{1},\ldots,\xi_N$ be the poles of $F_{\alpha}$ in $D_{\alpha}(\textup{\textbf{F}},\textup{\textbf{m}}).$  For each $j=1,\ldots,N,$ let $\hat{\tau}_j$ be the order of $\xi_j$ as a pole of $F_{\alpha}$ and $\tau_j$ its order as a system pole. By assumption, $\hat{\tau}_j\leq \tau_j.$ Set
$$\boldsymbol\rho_{\alpha}^{*}(\textup{\textbf{F}},\textup{\textbf{m}}):=\min\{ \boldsymbol\rho_{\alpha}(\textup{\textbf{F}},\textup{\textbf{m}}),\min_{j=1,\ldots,N} \boldsymbol\rho_{\xi_j,\hat{\tau}_j}(\textup{\textbf{F}},\textup{\textbf{m}})\}$$
and let $D_{\alpha}^{*}(\textup{\textbf{F}},\textup{\textbf{m}})$ be the canonical domain with this index. We have assumed that $0 \in E$ where all the functions $F_\alpha$ are holomorphic; consequently, for a fixed $\alpha$ if we were to define an analogous quantity for $z^kF_\alpha$ we would obtain the same
number $\boldsymbol\rho_{\alpha}^{*}(\textup{\textbf{F}},\textup{\textbf{m}})$ independently of $k$.

\medskip

By $Q_{\textup{\textbf{m}}}^{\textup{\textbf{F}}},$ we denote the monic polynomial whose zeros are the system poles of $\textup{\textbf{F}}$ with respect to $\textup{\textbf{m}}$ taking account of their order. The set of distinct zeros of $Q_{\textup{\textbf{m}}}^{\textup{\textbf{F}}}$ is denoted by $\mathcal{P}(\textup{\textbf{F}},\textup{\textbf{m}}).$

\medskip

We are ready to state the direct result.

\begin{thm}\label{thm1.4} Let $\textup{\textbf{F}}=(F_1,\ldots,F_d)\in \mathcal{H}(E)^d$ and let  $\textup{\textbf{m}}\in \mathbb{N}^d$ be a fixed multi-index. Suppose that $\textup{\textbf{F}}$ has exactly $|\textup{\textbf{m}}|$ system poles with respect to $\textup{\textbf{m}}$ counting multiplicities. Then, for all  sufficiently large $n$, the polynomials $Q_{n,\textup{\textbf{m}}}$ and the approximants $R_{n,\textup{\textbf{m}},\alpha}$ are uniquely determined,
\begin{equation}\label{2.5}
\limsup_{n \rightarrow \infty} \|Q_{n,\textup{\textbf{m}}}-Q_{\textup{\textbf{m}}}^{\textup{\textbf{F}}}\|^{1/n}=
\max \left\{\frac{|\Phi(\xi)|}{\boldsymbol\rho_{\xi}(\textup{\textbf{F}},\textup{\textbf{m}})} :\xi\in\mathcal{P}(\textup{\textbf{F}},\textup{\textbf{m}})\right\},
\end{equation}
where $\|\cdot\|$ denotes the coefficient norm in the space of polynomials. For any $\alpha=1,\ldots,d,$ $k=1,\ldots,m_\alpha -1$, and
any compact subset $K$ of $D_{\alpha}^{*}(\textup{\textbf{F}},\textup{\textbf{m}})\setminus \mathcal{P}(\textup{\textbf{F}},\textup{\textbf{m}}),$
\begin{equation}\label{approximation}
\limsup_{n \rightarrow \infty} \left\|\frac{P_{n,\textup{\textbf{m}},k,\alpha}}{Q_{n,\textup{\textbf{m}}}}-z^{k}F_{\alpha}\right\|_{K}^{1/n}\leq \frac{\|\Phi\|_K}{\boldsymbol\rho_{\alpha}^{*}(\textup{\textbf{F}},\textup{\textbf{m}})}.
\end{equation}
where $\|\cdot \|_{K}$ denotes the sup-norm on $K$ and if $K\subset E,$ then $\|\Phi\|_K$ is replaced by $1.$
\end{thm}

In the inverse direction, we have

\begin{thm}\label{inverse} Let $\textup{\textbf{F}}=(F_1,F_2,\ldots,F_d)\in \mathcal{H}(E)^d$ and  $\textup{\textbf{m}}\in \mathbb{N}^d$ be a fixed multi-index.   Suppose that the polynomials $Q_{n,\textup{\textbf{m}}}$ are uniquely determined for all sufficiently large $n$ and there exists a polynomial $Q_{|\textup{\textbf{m}}|}$ of degree $|\textup{\textbf{m}}|$ such that
$$\limsup_{n \rightarrow \infty} \|Q_{n,\textup{\textbf{m}}}-Q_{|\textup{\textbf{m}}|}\|^{1/n}=\theta<1.$$
Then, $\textup{\textbf{F}}$ has exactly $|\textup{\textbf{m}}|$ system poles with respect to $\textup{\textbf{m}}$ counting multiplicities and $Q_{|\textup{\textbf{m}}|}=Q_{\textup{\textbf{m}}}^{\textup{\textbf{F}}}.$
\end{thm}

An immediate consequence of Theorems \ref{thm1.4} and \ref{inverse} is the following corollary which is the analogue of the Montessus de Ballore-Gonchar theorem for simultaneous Pad\'{e}-Faber approximation.

\begin{corollary} Let $\textup{\textbf{F}}=(F_1,F_2,\ldots,F_d)\in \mathcal{H}(E)^d$ and $\textup{\textbf{m}}\in \mathbb{N}^d$ be a fixed multi-index.  Then, the following assertions are equivalent:
\begin{enumerate}
\item [(a)] $\textup{\textbf{F}}$ has exactly $|\textup{\textbf{m}}|$ system poles with respect to $\textup{\textbf{m}}$ counting multiplicities.
\item [(b)] The polynomials $Q_{n,\textup{\textbf{m}}}$ of $\textup{\textbf{F}}$ are uniquely determined for all sufficiently large $n$ and there exists a polynomial $Q_{|\textup{\textbf{m}}|}$ of degree $|\textup{\textbf{m}}|$ such that
$$\limsup_{n \rightarrow \infty} \|Q_{n,\textup{\textbf{m}}}-Q_{|\textup{\textbf{m}}|}\|^{1/n}=\theta<1.$$
\end{enumerate}
Consequently, if either (a) or (b) takes place, then $Q_{|\textup{\textbf{m}}|}=Q_{\textup{\textbf{m}}}^{\textup{\textbf{F}}},$ and \eqref{2.5}-\eqref{approximation} hold.
\end{corollary}

The outline of this paper is as follows. Section 2 contains the proof of Theorem \ref{thm1.4}. The proof of Theorem \ref{inverse} is in Section 3.

\section{Proof of Theorem \ref{thm1.4}}

\subsection{Auxiliary Lemmas}

The following lemma (see, e.g.,  \cite{SmirnovLebedev} or \cite{Suetin}) is obtained using \eqref{conv} the same way as similar statements are proved for Taylor series.
\begin{lemma}\label{expan} Let $G\in \mathcal{H}(E)$. Then,
\begin{equation}\label{defofrhomF}
\rho_0(G)=\left(\limsup_{n \rightarrow \infty} |[G]_n|^{1/n} \right)^{-1}.
\end{equation}
Moreover, $\sum_{n=0}^{\infty} [G]_n \Phi_n(z)$ converges to $G(z)$ uniformly inside ${D}_{\rho_{0}(G)}.$ \end{lemma}
Here and in what follows, the phrase ``uniformly inside a domain" means ``uniformly on each compact subset of the domain".

\medskip

As a consequence of Lemma \ref{expan},
if  $\textup{\textbf{F}}=(F_1,F_2,\ldots,F_d)\in \mathcal{H}(E)^d$, then for each $\alpha=1,2,\ldots,d$ and $k=0,1,\ldots,m_{\alpha}-1$ fixed,
\begin{equation}\label{usethisasdef}
z^{k}Q_{n,\textup{\textbf{m}}}(z)F_{\alpha}(z)-P_{n,\textup{\textbf{m}},k,\alpha}(z)=\sum_{\ell=n+1}^{\infty} [z^{k} Q_{n,\textup{\textbf{m}}}F_{\alpha}]_{\ell}\Phi_{\ell}(z), \quad \quad z\in D_{\rho_{0}(F_{\alpha})},
\end{equation}
and $P_{n,\textup{\textbf{m}},k,\alpha}=\sum_{\ell=0}^{n-1} [ z^{k} Q_{n,\textup{\textbf{m}}} F_{\alpha}]_{\ell} \Phi_\ell$ is uniquely determined by $Q_{n,\textup{\textbf{m}}}.$

\medskip

The next lemma (see \cite[p. 583]{Curtiss} or \cite[p. 43]{Suetin} for its proof) gives an estimate of Faber polynomials $\Phi_n$ on a level curve.
\begin{lemma}\label{estimate}
Let $\rho>1$ be fixed.
Then,  there exists $c>0$ such that
\begin{equation}\label{estimateFaber}
\|\Phi_n\|_{\Gamma_{\rho}}\leq c \rho^n, \qquad n\geq 0.
\end{equation}
\end{lemma}

\subsection{Proof of Theorem \ref{thm1.4}}

\begin{proof}[Proof of Theorem \ref{thm1.4}]

For each $n\in \mathbb{N},$ let  $q_{n,\textup{\textbf{m}}}$ be the polynomial $Q_{n,\textup{\textbf{m}}}$ normalized so that
\begin{equation}\label{contradict}
q_{n,\textup{\textbf{m}}}(z)=\sum_{k=0}^{|\textup{\textbf{m}}|} \lambda_{n,k} z^{k}, \qquad \sum_{k=0}^{|\textup{\textbf{m}}|} |\lambda_{n,k}|=1.
\end{equation}
With this normalization, the polynomials $q_{n,\textup{\textbf{m}}}$ are uniformly bounded on each compact subset of $\mathbb{C}.$

\medskip

Let $\xi$ be a system pole of order $\tau$  of $\textup{\textbf{F}}$ with respect to $\textup{\textbf{m}}.$
We will show that
\begin{equation}\label{3.31}
\limsup_{n \rightarrow \infty} |q_{n,\textup{\textbf{m}}}^{(j)}(\xi)|^{1/n}\leq \frac{|\Phi(\xi)|}{\boldsymbol \rho_{\xi,j+1}(\textup{\textbf{F}},\textup{\textbf{m}})}, \quad \quad j=0,1,\ldots,\tau-1.
\end{equation} 

Fix $\ell \in \{1,\ldots,\tau\}$. Consider a polynomial combination of $G_\ell$ of the type \eqref{polycom} that is holomorphic on a neighborhood of $\overline{D}_{|\Phi(\xi)|}$ except for a pole of order $\ell$ at $z=\xi$ and verifies that $\rho_{\ell}(G_\ell)= \rho_{\xi,\ell}(\textup{\textbf{F}},\textup{\textbf{m}}).$ Then, we have
$$G_\ell=\sum_{\alpha=1}^d v_{\alpha,\ell} F_{\alpha}, \quad \quad \deg v_{\alpha,\ell}< m_{\alpha}, \quad \alpha=1,2,\ldots,d.$$
Set $$H_\ell(z):=(x-\xi)^{\ell} G_\ell(z) \quad \quad \textup{and}\quad \quad a_{n,n}^{(\ell)}:=[q_{n,\textup{\textbf{m}}}G_\ell]_n.$$ By the definition of $Q_{n,\textup{\textbf{m}}}$, it follows that
$a_{n,n}^{(\ell)}=0.$
Therefore,
 $$a_{n,n}^{(\ell)}=[q_{n,\textup{\textbf{m}}}G_\ell]_n=\frac{1}{2\pi i} \int_{\Gamma_{\rho_1}} \frac{q_{n,\textup{\textbf{m}}}(z)G_\ell(z) \Phi'(z)}{\Phi^{n+1}(z)}dz = 0,$$
where $1< \rho_1< |\Phi(\xi)|.$ Set
 $$\tau_{n,n}^{(\ell)}:=\frac{1}{2\pi i} \int_{\Gamma_{\rho_2}} \frac{q_{n,\textup{\textbf{m}}}(z)G_\ell(z) \Phi'(z)}{\Phi^{n+1}(z)}dz,$$
 where $|\Phi(\xi)|<\rho_2<\rho_{\xi,\ell}(\textup{\textbf{F}},\textup{\textbf{m}}).$ Using  Cauchy's residue theorem on the function $(q_{n,\textup{\textbf{m}}}G_\ell \Phi')/{\Phi^{n+1}}$, we obtain
$$
\tau_{n,n}^{(\ell)}=\tau_{n,n}^{(\ell)}-a_{n,n}^{(\ell)}=\frac{1}{2 \pi i}\int_{\Gamma_{\rho_2}}  \frac{q_{n,\textup{\textbf{m}}}(t)G_\ell(t) \Phi'(t)}{\Phi^{n+1}(t)}dt-\frac{1}{2 \pi i} \int_{\Gamma_{\rho_1}}\frac{q_{n,\textup{\textbf{m}}}(t)G_\ell(t) \Phi'(t)}{\Phi^{n+1}(t)} dt
$$
\begin{equation}\label{reduce}
= \textup{res}((q_{n,\textup{\textbf{m}}}G_\ell \Phi')/{\Phi^{n+1}},\, \xi).
\end{equation}
Now,
$$
\textup{res}((q_{n,\textup{\textbf{m}}}G_\ell \Phi')/{\Phi^{n+1}},\, \xi)=\frac{1}{(\ell-1)!} \lim_{z \rightarrow \xi} \left(\frac{(z-\xi)^{\ell} G_\ell(z)  \Phi'(z) q_{n,\textup{\textbf{m}}}(z)}{\Phi^{n+1}(z)}\right)^{(\ell-1)}
$$
\begin{equation}\label{reduce111111}
=\frac{1}{(\ell-1)!} \sum_{t=0}^{\ell-1} {\ell-1 \choose t }  \left(\frac{H_\ell \Phi'}{\Phi^{n+1}} \right)^{(\ell-1-t)}(\xi) q_{n,\textup{\textbf{m}}}^{(t)}(\xi).
\end{equation}
Consequently,
$$(\ell-1)!\tau_{n,n}^{(\ell)}= \left(\frac{H_\ell(\xi) \Phi'(\xi)}{\Phi^{n+1}(\xi)}\right) q_{n,\textup{\textbf{m}}}^{(\ell-1)}(\xi) + \sum_{t=0}^{\ell-2} {\ell-1 \choose t }  \left(\frac{H_\ell \Phi'}{\Phi^{n+1}} \right)^{(\ell-1-t)}(\xi) q_{n,\textup{\textbf{m}}}^{(t)}(\xi),
$$
where the sum is empty when $\ell = 1$. Therefore,
\begin{equation}\label{use678}
q_{n,\textup{\textbf{m}}}^{(\ell-1)}(\xi)=\frac{(\ell-1)!\tau_{n,n}^{(\ell)} \Phi^{n+1}(\xi)}{H_\ell(\xi)\Phi'(\xi) } -\sum_{t=0}^{\ell-2} {\ell-1 \choose t }  \left(\frac{{H_\ell \Phi'}}{{\Phi^{n+1}}} \right)^{(\ell-1-t)}(\xi)\frac{  \Phi^{n+1}(\xi) q_{n,\textup{\textbf{m}}}^{(t)}(\xi)}{H_\ell(\xi)\Phi'(\xi) }.
\end{equation}

Choose $\delta>0$  small enough so that
\begin{equation}\label{yuhyggt1*}
\rho_2:=\rho_{\xi,\ell}(\textup{\textbf{F}},\textup{\textbf{m}})-\delta>|\Phi(\xi)|.
\end{equation}
We have
\begin{equation}\label{use1y}
|\tau_{n,n}^{(\ell)}|=\left|\frac{1}{2\pi i} \int_{\Gamma_{\rho_2}} \frac{q_{n,\textup{\textbf{m}}}(z)G_\ell(z) \Phi'(z)}{\Phi^{n+1}(z)}dz\right|\leq \frac{c_1}{\rho_2^n}.
\end{equation}
If $\ell = 1$, from \eqref{use678} and \eqref{use1y} we obtain
$$|q_{n,\textup{\textbf{m}}}(\xi)|\leq c_2 \left( \frac{|\Phi(\xi)|}{\rho_2}\right)^n$$ which implies that
$$\limsup_{n \rightarrow \infty}|q_{n,\textup{\textbf{m}}}(\xi)|^{1/n} \leq \frac{|\Phi(\xi)|}{\rho_2}.$$
Letting $\delta \rightarrow 0,$ \eqref{3.31} readily follows for $j=0$. For the remaining values of $j,$ we use induction.

\medskip

Suppose that \eqref{3.31} is true for $j=0,\ldots,\ell -2, 2 \leq \ell \leq \tau$ and let us prove that it is also valid for $j=\ell-1$. Choosing $\delta > 0$ as in \eqref{yuhyggt1*}, for $t=0,1,\ldots,\ell-2,$ we obtain
\begin{equation}\label{use3y}
\left|\left(\frac{H_\ell \Phi'}{ \Phi^{n+1}}\right)^{(\ell-1-t)}(\xi)\right|=\left|\frac{(\ell-1-t)!}{2\pi i} \int_{|z-\xi|=\varepsilon}\frac{H_\ell(z) \Phi'(z)}{(z-\xi)^{\ell-t} \Phi^{n+1}(z)}dz\right|\leq \frac{c_3}{(|\Phi(\xi)|-\delta)^n},
\end{equation}
where $\{z\in \mathbb{C}: |z-\xi|=\varepsilon\}\subset \{z\in \mathbb{C}: |\Phi(z)|>|\Phi(\xi)|-\delta\}.$
Combining the induction hypothesis, \eqref{use678}, \eqref{use1y}, and \eqref{use3y}, it follows from \eqref{use678} that
$$\limsup_{n \rightarrow \infty}\left|(q_{n,\textup{\textbf{m}}})^{(\ell-1)}(\xi)\right|^{1/n}$$
$$=\limsup_{n \rightarrow \infty}\left|\frac{(\ell-1)!\tau_{n,n}^{(\ell)} \Phi^{n+1}(\xi)}{H_\ell(\xi)\Phi'(\xi) }-\sum_{t=0}^{\ell-2} {\ell-1 \choose t } \left(\frac{{H_\ell \Phi'}}{{\Phi^{n+1}}} \right)^{(\ell-1-t)}(\xi)\frac{ \Phi^{n+1}(\xi) (q_{n,\textup{\textbf{m}}})^{(t)}(\xi)}{H_\ell(\xi)\Phi'(\xi) }\right|^{1/n}$$
\begin{equation}\label{have}
 \leq  \max \left\{ \frac{|\Phi(\xi)|}{\rho_2}, \left( \frac{|\Phi(\xi)|}{|\Phi(\xi)|-\delta}\right)\left(\frac{|\Phi(\xi)|}{\boldsymbol \rho_{\xi,\ell-1}(\textup{\textbf{F}},\textup{\textbf{m}})}\right) \right\}.
\end{equation}
Letting $\delta \rightarrow 0,$ we have $\rho_2\rightarrow  \rho_{\xi,\ell}(\textup{\textbf{F}},\textup{\textbf{m}})$ and from \eqref{have}, we obtain
$$\limsup_{n \rightarrow \infty} \left|(q_{n,\textup{\textbf{m}}})^{(\ell-1)}(\xi)\right|^{1/n}\leq  \max \left\{\frac{|\Phi(\xi)|}{\rho_{\xi,\ell}(\textup{\textbf{F}},\textup{\textbf{m}})} , \frac{|\Phi(\xi)|}{ \boldsymbol \rho_{\xi,\ell-1}(\textup{\textbf{F}},\textup{\textbf{m}})} \right\} \leq \frac{|\Phi(\xi)|}{\boldsymbol \rho_{\xi,\ell}(\textup{\textbf{F}},\textup{\textbf{m}})}.
$$
which completes the induction.

\medskip

Let $\xi_1,\ldots,\xi_w$ be the distinct system poles of $\textup{\textbf{F}}$ and let $\tau_j$ be the order of $\xi_j$ as a
system pole, $j=1,\ldots,w.$ By assumption, $\tau_1+\ldots+\tau_w =|\textup{\textbf{m}}|.$ We have proved that for $j =1,\ldots,w$ and $t =0,1,\ldots,\tau_j-1,$
\begin{equation}\label{3.32}
\limsup_{n \rightarrow \infty} |q_{n,\textup{\textbf{m}}}^{(t)}(\xi_j)|^{1/n}\leq \frac{|\Phi(\xi_j)|}{\boldsymbol \rho_{\xi_j,t+1}(\textup{\textbf{F}},\textup{\textbf{m}})}\leq \frac{|\Phi(\xi_j)|}{\boldsymbol \rho_{\xi_j} (\textup{\textbf{F}},\textup{\textbf{m}})}.
\end{equation}

Let $L_{j,t}, j=1,\ldots,w, t =0,1,\ldots,\tau_j-1,$ be the basis of polynomials of degree $\leq |{\bf m}| -1$ defined by the interpolation conditions
\[L_{j,t}^{(s)}(\xi_k) = \delta_{j,k} \delta_{t,s}, \qquad k=1,\ldots,w, \qquad s = 0,\ldots,\tau_k-1.\]
Then
\[q_{n,{\bf m}} - \lambda_{n,|{\bf m}|} Q^{\bf F}_{{\bf m}} = \sum_{j=1}^{w}\sum_{t=0}^{\tau_j -1} q_{n,{\bf m}}^{(t)}(\xi_j)L_{j,t},\]
where $\lambda_{n,|{\bf m}|}$ is the leading coefficient of $q_{n,{\bf m}}$. From \eqref{3.32} it follows that
\[\limsup_{n \rightarrow \infty} \|q_{n,\textup{\textbf{m}}}-\lambda_{n,|{\bf m}|}  Q_{\textup{\textbf{m}}}^{{\textup{\textbf{F}}}}\|_K^{1/n}\leq \max \left\{\frac{|\Phi(\xi)|}{\boldsymbol\rho_{\xi}(\textup{\textbf{F}},\textup{\textbf{m}})} :\xi\in\mathcal{P}(\textup{\textbf{F}},\textup{\textbf{m}})\right\}\]
for every compact set $K\subset \mathbb{C}$.  In finite dimensional spaces all norms are equivalent; therefore, \begin{equation}\label{usenewone}
\limsup_{n \rightarrow \infty} \|q_{n,\textup{\textbf{m}}}-\lambda_{n,|{\bf m}|} Q_{\textup{\textbf{m}}}^{{\textup{\textbf{F}}}}\|^{1/n}\leq \max \left\{\frac{|\Phi(\xi)|}{\boldsymbol\rho_{\xi}(\textup{\textbf{F}},\textup{\textbf{m}})} :\xi\in\mathcal{P}(\textup{\textbf{F}},\textup{\textbf{m}})\right\}.
\end{equation}
In turn, this implies that
\begin{equation}\label{need2}
\liminf_{n \rightarrow \infty} |\lambda_{n,|{\bf m}|} |>0,
\end{equation}
since otherwise for a subsequence of indices $\Lambda,$ we would have $\lim_{n\in \Lambda} \|q_{n,{\bf m}}\| = 0$ which contradicts the normalization imposed on the polynomials $Q_{n,{\bf m}}$ (see \eqref{contradict}). Combining \eqref{usenewone} and \eqref{need2}, we get \eqref{2.5} with $\leq$ in place of $=$.

\medskip

Now we know that $\deg Q_{n,\textup{\textbf{m}}}=|\textup{\textbf{m}}|, n \geq n_0,$ since these polynomials converge to a polynomial of degree $|{\bf m}|$. In turn, this implies that $Q_{n,{\bf m}}$ is uniquely determined for all sufficiently large $n$ because the difference of any two distinct monic polynomials satisfying Definition \ref{simu} with the same degree produces a new solution of degree strictly less than $|\textup{\textbf{m}}|$, but we have proved that any solution must have degree $|\textup{\textbf{m}}|$ for all sufficiently large $n.$  Definition \ref{simu} implies that $P_{n,{\bf m},0,\alpha}$ is determined uniquely through $Q_{n,{\bf m}}$; consequently $R_{n,{\bf m},\alpha}$ is uniquely determined for all large enough $n$.

\medskip

Now, we prove the equality in \eqref{2.5}. To the contrary, suppose that
\begin{equation}\label{less}
\limsup_{n \rightarrow \infty} \|Q_{n,\textup{\textbf{m}}}-Q_{\textup{\textbf{m}}}^{\textup{\textbf{F}}}\|^{1/n}=\theta<
\max \left\{\frac{|\Phi(\xi)|}{\rho_{\xi}(\textup{\textbf{F}},\textup{\textbf{m}})} :\xi\in\mathcal{P}(\textup{\textbf{F}},\textup{\textbf{m}})\right\}.
\end{equation}
Let $\zeta$ be a system pole of $\textup{\textbf{F}}$ such that
$$\frac{|\Phi(\zeta)|}{\boldsymbol \rho_{\zeta}(\textup{\textbf{F}},\textup{\textbf{m}})}=\max \left\{\frac{|\Phi(\xi)|}{\rho_{\xi}(\textup{\textbf{F}},\textup{\textbf{m}})} :\xi\in\mathcal{P}(\textup{\textbf{F}},\textup{\textbf{m}})\right\}.$$ Clearly, the inequality \eqref{less} implies that $\boldsymbol \rho_{\zeta}(\textup{\textbf{F}},\textup{\textbf{m}})<\infty.$

\medskip

Choose a polynomial combination
\begin{equation}\label{uhbjikml}
G=\sum_{\alpha=1}^d v_{\alpha} F_{\alpha},\quad \quad \deg v_{\alpha} < m_{\alpha},\quad\quad  \alpha=1,2,\ldots,d,
\end{equation}
that is holomorphic on a neighborhood of $\overline{D}_{|\Phi(\zeta)|}$ except for a pole of order $s$ at $z=\zeta$ with $\rho_s(G)=\boldsymbol \rho_{\zeta}(\textup{\textbf{F}},\textup{\textbf{m}}).$
Notice that $Q_{\textup{\textbf{m}}}^{\textup{\textbf{F}}}G$ must have a singularity on the boundary of $D_{\rho_s}(G)$ which implies
\begin{equation}\label{newimprove}
\limsup_{n \rightarrow \infty} |[Q_{\textup{\textbf{m}}}^{\textup{\textbf{F}}}G]_n|^{1/n}=\frac{1}{\boldsymbol \rho_{\zeta}(\textup{\textbf{F}},\textup{\textbf{m}})}.
\end{equation}
In fact, if $Q_{\textup{\textbf{m}}}^{\textup{\textbf{F}}}G$ had no singularity on the boundary of $D_{\rho_s}(G)$, then all singularities of $G$  on the boundary of $D_{\rho_s}(G)$ would be at most poles and their order as poles of $G$ would be smaller than their order as system poles of ${\textup{\textbf{F}}}$. In this case, we could find a different polynomial combination $G_1$ of type \eqref{uhbjikml} for which $\rho_s(G_1)>\rho_s(G)=\boldsymbol \rho_{\zeta}(\textup{\textbf{F}},\textup{\textbf{m}})$ which contradicts the definition of $\boldsymbol \rho_{\zeta}(\textup{\textbf{F}},\textup{\textbf{m}})$. Therefore, $Q_{\textup{\textbf{m}}}^{\textup{\textbf{F}}}G$ has a singularity on the the boundary of $D_{\rho_s}(G)$ and the equality \eqref{newimprove} holds.

\medskip

Choose $1<\rho<|\Phi(\zeta)|$.
Then, by the definition of $Q_{n,\textup{\textbf{m}}},$ \eqref{less}, and \eqref{newimprove},
$$\frac{1}{\boldsymbol \rho_{\zeta}(\textup{\textbf{F}},\textup{\textbf{m}})}=\limsup_{n \rightarrow \infty} |[Q_{\textup{\textbf{m}}}^{\textup{\textbf{F}}}G]_n|^{1/n}=\limsup_{n \rightarrow \infty} |[Q_{\textup{\textbf{m}}}^{\textup{\textbf{F}}}G-Q_{n,\textup{\textbf{m}}}G]_n|^{1/n}$$
$$=\limsup_{n \rightarrow \infty}\left|\frac{1}{2\pi i} \int_{\Gamma_{\rho}} \frac{(Q_{\textup{\textbf{m}}}^{\textup{\textbf{F}}}-Q_{n,\textup{\textbf{m}}})(z) G(z) \Phi'(z)}{\Phi^{n+1}(z)} dz \right|^{1/n}\leq \frac{\theta}{\rho}.$$
Letting  $\rho\rightarrow |\Phi(\zeta)|$ in the above inequality, we obtain the contradiction
$$\frac{1}{\boldsymbol \rho_{\zeta}(\textup{\textbf{F}},\textup{\textbf{m}})}\leq \frac{\theta}{|\Phi(\zeta)|}< \frac{|\Phi(\zeta)|/\boldsymbol \rho_{\zeta}(\textup{\textbf{F}},\textup{\textbf{m}})}{|\Phi(\zeta)|}=\frac{1}{\boldsymbol \rho_{\zeta}(\textup{\textbf{F}},\textup{\textbf{m}})}.$$  

\medskip

Let us prove the inequality \eqref{approximation}. Let $\alpha\in \{1,\ldots,d\}$ and $k\in \{0,1,\ldots,m_{\alpha}-1\}$ be fixed and let $\tilde{\xi}_1,\ldots,\tilde{\xi}_N$ be the poles of $z^{k}F_{\alpha}$ in $D_{\alpha}(\textup{\textbf{F}},\textup{\textbf{m}}).$ For each $j=1,\ldots,N,$ let $\hat{\tau}_j$ be the order of $\tilde{\xi}_j$ as a pole of $z^{k}F_{\alpha}$ and $\tilde{\tau}_j$ its order as a system pole of $\textbf{\textup{F}}$. Recall that by assumption, $\hat{\tau}_j \leq\tilde{\tau}_j.$ From equation \eqref{usethisasdef}, we have
$$ Q_{n,\textup{\textbf{m}}} z^{k}  F_{\alpha}-P_{n,\textup{\textbf{m}},k,\alpha}=\sum_{\ell=n+1}^{\infty} a_{\ell,n} \Phi_{\ell}.$$
Multiplying the above equality by $\omega(z):=\prod_{j=1}^N(z-\tilde{\xi}_j)^{\hat{\tau}_j}$ and expanding the result in terms of the Faber polynomial expansion, we obtain
$$\omega Q_{n,\textup{\textbf{m}}} z^{k}F_{\alpha}-\omega P_{n,\textup{\textbf{m}},k,\alpha}=\sum_{\ell=n+1}^{\infty} a_{\ell,n} \omega \Phi_{\ell}$$
\begin{equation}\label{mainmain3}
=\sum_{\nu=0}^{\infty} b_{\nu,n} \Phi_{\nu}=\sum_{\nu=0}^{n+|\textup{\textbf{m}}|} b_{\nu,n} \Phi_{\nu}+\sum_{\nu=n+|\textup{\textbf{m}}|+1}^{\infty} b_{\nu,n} \Phi_{\nu}
\end{equation}

Let $K$ be a compact subset of $D_{\alpha}^{*}{({\textup{\textbf{F}},\textup{\textbf{m}}})}\setminus \mathcal{P}(\textup{\textbf{F}},\textup{\textbf{m}})$ and set
\begin{equation}\label{tgbrdcvhjkl}
 \sigma:=\max\{\|\Phi\|_K,1\}
 \end{equation}
($\sigma=1$ when $K \subset E$). Choose $\delta>0$ so small that
\begin{equation}\label{yuhyggt1}
\rho_2:=\boldsymbol \rho_{\alpha}^{*}({\textup{\textbf{F}},\textup{\textbf{m}}})-\delta>\sigma.
\end{equation}

\medskip

Let us estimate $\sum_{\nu=n+|\textup{\textbf{m}}|+1}^{\infty} |b_{\nu,n}| |\Phi_{\nu}|$ on $\overline{D}_{\sigma}.$ For $\nu \geq n+ |\textup{\textbf{m}}|+1,$
$$b_{\nu,n}:=[\omega Q_{n,\textup{\textbf{m}}} z^{k}F_{\alpha}-\omega P_{n,\textup{\textbf{m}},k,\alpha}]_{\nu}=[\omega Q_{n,\textup{\textbf{m}}} z^{k}F_\alpha]_{\nu}$$
$$=\frac{1}{2\pi i}\int_{\Gamma_{\rho_2}} \frac{z^{k}\omega(z) Q_{n,\textup{\textbf{m}}}(z) F_{\alpha}(z)\Phi'(z)} {\Phi^{\nu+1}(z)} dz,$$
where  $1< \rho_2<  \boldsymbol \rho^{*}_{\alpha}(\textup{\textbf{F}},\textup{\textbf{m}}).$ By a computation similar to \eqref{use1y}, we obtain
\begin{equation}\label{secondtypeasym}
|b_{\nu,n}|\leq  \frac{c_{4}}{ \rho_2^\nu}.
\end{equation}

Combining \eqref{yuhyggt1}, \eqref{secondtypeasym}, and Lemma \ref{estimate}, we have for $z\in \overline{D}_{\sigma},$
$$\sum_{\nu=n+|\textup{\textbf{m}}|+1}^{\infty} |b_{\nu,n}| |\Phi_{\nu}(z)|\leq c_{5} \sum_{\nu=n+|\textup{\textbf{m}}|+1}^{\infty}\left( \frac{\sigma}{\rho_2}\right)^{\nu}=c_{6} \left( \frac{\sigma}{\rho_2}\right)^{n},$$ which implies that
$$\limsup_{n \rightarrow \infty}\left\|\sum_{\nu=n+|\textup{\textbf{m}}|+1}^{\infty} |b_{\nu,n}| |\Phi_{\nu}|\right\|_{\overline{D}_{\sigma}}^{1/n}\leq  \frac{\sigma}{\rho_2}.$$
Letting $\delta \rightarrow 0^{+},$ we have $\rho_2\rightarrow \boldsymbol \rho_{\alpha}^{*}(\textup{\textbf{F}},\textup{\textbf{m}})$ and
\begin{equation}\label{estimate1}
\limsup_{n \rightarrow \infty}\left\|\sum_{\nu=n+|\textup{\textbf{m}}|+1}^{\infty} |b_{\nu,n}| |\Phi_{\nu}|\right\|_{\overline{D}_{\sigma}}^{1/n}\leq  \frac{\sigma}{\boldsymbol \rho_{\alpha}^{*}(\textup{\textbf{F}},\textup{\textbf{m}})}.
\end{equation}

Now, we wish to estimate $\sum_{\nu=0}^{n+|\textup{\textbf{m}}|} |b_{\nu,n}| |\Phi_{\nu}|$ on $\overline{D}_{\sigma}.$ Notice that
$$b_{\nu,n}=\sum_{\ell=n+1}^{\infty} a_{\ell,n}[ \omega  \Phi_{\ell}]_{\nu}.$$ Therefore,   we need to estimate both
$|a_{\ell,n}|$ and $|[ \omega  \Phi_{\ell}]_{\nu}|$.

\medskip

First, we work on $|a_{\ell,n}|$.
Combining \eqref{3.32} and \eqref{need2}, it follows that for the system poles $\xi_1,\ldots,\xi_w$ of $\textup{\textbf{F}}$, if $\tau_j$ is the order (as a system pole) of $\xi_j,$ then
\begin{equation}\label{3.31111}
\limsup_{n \rightarrow \infty} |Q_{n,\textup{\textbf{m}}}^{(u)}(\xi_j)|^{1/n}\leq \frac{|\Phi(\xi_j)|}{\boldsymbol \rho_{\xi_j,u+1}(\textup{\textbf{F}},\textup{\textbf{m}})}, \quad \quad u=0,1,\ldots,\tau_j-1.
\end{equation}
We have
$$a_{\ell,n}=[Q_{n,\textup{\textbf{m}}} z^{k}F_{\alpha}]_{\ell}=\frac{1}{2\pi i} \int_{\Gamma_{\rho_1}} \frac{Q_{n,\textup{\textbf{m}}}(z) z^{k}F_{\alpha}(z) \Phi'(z)}{\Phi^{\ell+1}(z)} dz,$$
where $1< \rho_1< \rho_{0}(z^{k}F_{\alpha})$ and define
$$\tau_{\ell,n}:=\frac{1}{2\pi i} \int_{\Gamma_{\rho_2}} \frac{Q_{n,\textup{\textbf{m}}}(z) z^{k}F_{\alpha}(z) \Phi'(z)}{\Phi^{\ell+1}(z)} dz,$$
where $\max\{|\Phi(\tilde{\xi}_j)|:j=1,\ldots,N\}< \rho_2< \boldsymbol \rho_{\alpha}^{*}(\textup{\textbf{F}},\textup{\textbf{m}}).$
Arguing as in \eqref{reduce} and \eqref{reduce111111}, we obtain
$$
\tau_{\ell,n}-a_{\ell,n}=\sum_{j=1}^N\textup{res}(Q_{n,\textup{\textbf{m}}} z^{k} F_{\alpha} \Phi'/\Phi^{\ell+1},\, \tilde{\xi}_j)
$$
\begin{equation}\label{qwertyuiopsdfghjk}
=\sum_{j=1}^N \frac{1}{(\hat{\tau}_j-1)!} \sum_{u=0}^{\hat{\tau}_j-1} {\hat{\tau}_j-1 \choose u }  \left(\frac{(z-\tilde{\xi}_j)^{\hat{\tau}_j}z^{k}F_{\alpha}\Phi'}{\Phi^{\ell+1}}\right)^{(\hat{\tau}_j-1-u)}(\tilde{\xi}_j) Q_{n,\textup{\textbf{m}}}^{(u)}(\tilde{\xi}_j).
\end{equation}
Notice that $(z-\tilde{\xi}_j)^{\hat{\tau}_j}z^{k}F_{\alpha}$ is holomorphic at $\tilde{\xi}_j.$  Let $\delta>0$ be such that
$$
|\Phi(\tilde{\xi}_j)|-2\delta>1, \quad \quad\quad \quad  j=1,\ldots, N.
$$
Computations similar to \eqref{use1y} and \eqref{use3y} give us
\begin{equation}\label{useuseuse}
|\tau_{\ell,n}|\leq \frac{c_{7}}{\rho_2^\ell}\quad \quad \textup{and}\quad \quad \left|\left(\frac{(z-\tilde{\xi}_j)^{\hat{\tau}_j}z^{k}F_{\alpha}\Phi'}  {\Phi^{\ell+1}}\right)^{(\hat{\tau}_j-1-u)}(\tilde{\xi}_j)\right|\leq \frac{c_{8}}{(|\Phi(\tilde{\xi}_j)|-\delta)^\ell},
\end{equation}
respectively.
Take $\varepsilon>0.$
From \eqref{3.31111} it follows that for all $j=1,\dots,N,$
$$|Q_{n,\textup{\textbf{m}}}^{(u)}(\tilde{\xi}_j)|\leq c_{9} \left(\frac{|\Phi(\tilde{\xi}_j)|+\varepsilon}{\boldsymbol \rho_{\tilde{\xi}_j,\hat{\tau}_j}(\textup{\textbf{F}},\textup{\textbf{m}})+\varepsilon} \right)^n.$$
Using \eqref{qwertyuiopsdfghjk}, \eqref{useuseuse} and the previous inequalities, we obtain
$$
|a_{\ell,n}| \leq |\tau_{\ell,n}|+
$$
$$\sum_{j=1}^N \sum_{u=0}^{\hat{\tau}_j-1} \frac{1}{(\hat{\tau}_j-1)!}  {\hat{\tau}_j-1 \choose u }  \left|\left(\frac{(z-\tilde{\xi}_j)^{\hat{\tau}_j}z^{k}F_{\alpha} \Phi'} {\Phi^{\ell+1}}\right)^{(\hat{\tau}_j-1-u)}(\tilde{\xi}_j)\right| \left| Q_{n,\textup{\textbf{m}}}^{(u)}(\tilde{\xi}_j)\right|
$$
$$
\leq \frac{c_{7}}{\rho_2^\ell}+ c_{10}  \sum_{j=1}^N \frac{ (|\Phi(\tilde{\xi}_j)|+\varepsilon)^n}{(\boldsymbol \rho_{\tilde{\xi}_j,\hat{\tau}_j}(\textup{\textbf{F}},\textup{\textbf{m}})+\varepsilon)^n(|\Phi(\tilde{\xi}_j)|-\delta)^\ell}
$$
\begin{equation}\label{approxaaa}
\leq \frac{c_{7}}{\rho_2^\ell}+ \frac{c_{10}}{( \boldsymbol \rho_{\alpha}^{*}(\textup{\textbf{F}},\textup{\textbf{m}})+\varepsilon)^n} \sum_{j=1}^N \frac{ (|\Phi(\tilde{\xi}_j)|+\varepsilon)^n}{(|\Phi(\tilde{\xi}_j)|-\delta)^\ell}
\end{equation}

Next, we estimate $| [\omega \Phi_\ell]_{\nu}|.$ We can assume that $\rho_1-\delta>1$. By Lemma \ref{estimate},
\begin{equation}\label{banana100hougeqs}
| [\omega \Phi_\ell]_{\nu}| \leq \left| \frac{1}{2\pi i}\int_{\Gamma_{\rho_1-\delta}} \frac{\omega(z) \Phi_{\ell}(z) \Phi'(z)}{ \Phi^{\nu+1}(z)} dz  \right|
\leq c_{11} \frac{(\rho_1-\delta)^{\ell}}{(\rho_1-\delta)^{\nu}}.
\end{equation}

By \eqref{approxaaa} and \eqref{banana100hougeqs}, we have
$$
|b_{\nu,n}|\leq \sum_{\ell=n+1}^{\infty} |a_{\ell,n}|| [\omega \Phi_\ell]_{\nu}|
$$
\begin{equation}\label{banana100hougeqs1}
\leq \frac{c_{12}}{(\rho_1-\delta)^{\nu}} \left(\frac{\rho_1-\delta}{\rho_2}\right)^n+ \frac{c_{13} (\rho_1-\delta)^n}{( \boldsymbol \rho_{\alpha}^{*}(\textup{\textbf{F}},\textup{\textbf{m}})+\varepsilon)^n {(\rho_1-\delta)^{\nu}}}  \sum_{j=1}^N  \left(\frac{ |\Phi(\tilde{\xi}_j)|+\varepsilon}{|\Phi(\tilde{\xi}_j)|-\delta}\right)^n.
\end{equation}

Combining \eqref{banana100hougeqs1} and Lemma \ref{estimate},  for $z\in \overline{D}_{\sigma}$ we obtain
$$\sum_{\nu=0}^{n+|\textup{\textbf{m}}|} |b_{\nu,n}| |\Phi_{\nu}(z)| \leq
$$
$$
\left(c_{14} \left(\frac{\rho_1-\delta}{\rho_2}\right)^n+ \frac{c_{15} (\rho_1-\delta)^n}{( \boldsymbol \rho_{\alpha}^{*}(\textup{\textbf{F}},\textup{\textbf{m}})+\varepsilon)^n }  \sum_{j=1}^N  \left(\frac{ |\Phi(\tilde{\xi}_j)|+\varepsilon}{|\Phi(\tilde{\xi}_j)|-\delta}\right)^n\right) \sum_{\nu=0}^{n+|\textup{\textbf{m}}|} \left(\frac{\sigma}{\rho_1-\delta} \right)^{\nu} \leq
$$
$$
\left(c_{14} \left(\frac{\rho_1-\delta}{\rho_2}\right)^n+\frac{c_{15}  (\rho_1-\delta)^n}{( \boldsymbol \rho_{\alpha}^{*}(\textup{\textbf{F}},\textup{\textbf{m}})+\varepsilon)^n }  \sum_{j=1}^N  \left(\frac{ |\Phi(\tilde{\xi}_j)|+\varepsilon}{|\Phi(\tilde{\xi}_j)|-\delta}\right)^n\right)({n+|\textup{\textbf{m}}|+1})  \sigma^{n+|\textup{\textbf{m}}|}.
$$
This implies that
$$\limsup_{n \rightarrow \infty} \left\|\sum_{\nu=0}^{n+|\textup{\textbf{m}}|} |b_{\nu,n}| |\Phi_{\nu}| \right\|_{\overline{D}_{\sigma}}^{1/n} \leq $$
$$\max \left\{  \frac{\sigma(\rho_1 -\delta) }{\rho_2}, \frac{\sigma(\rho_1-\delta)}{\boldsymbol \rho_{\alpha}^{*}(\textup{\textbf{F}},\textup{\textbf{m}})+\varepsilon} \max_{j=1,\ldots,N} \left(\frac{|\Phi(\tilde{\xi}_j)|+\varepsilon}{|\Phi(\tilde{\xi}_j)|-\delta}\right) \right\}.$$
Letting $\varepsilon,\delta\rightarrow 0^{+},$ and $\rho_1 \rightarrow 1^{+}$, we have $\rho_2 \rightarrow \boldsymbol \rho_{\alpha}^{*}(\textup{\textbf{F}},\textup{\textbf{m}})$ and we obtain
\begin{equation}\label{mainmain1}
\limsup_{n \rightarrow \infty} \left\|\sum_{\nu=0}^{n+|\textup{\textbf{m}}|} |b_{\nu,n}| |\Phi_{\nu}| \right\|_{\overline{D}_{\sigma}}^{1/n}\leq \frac{\sigma}{ \boldsymbol \rho_{\alpha}^{*}(\textup{\textbf{F}},\textup{\textbf{m}})}.
\end{equation}
Using \eqref{2.5}, \eqref{mainmain3}, \eqref{estimate1}, and \eqref{mainmain1}, we obtain \eqref{approximation} and the proof is complete.
\end{proof}

\section{Proof of Theorem \ref{inverse}}

\subsection{Incomplete Pad\'{e}-Faber approximation}\label{subsection3.1}

Let us introduce the notion of incomplete Pad\'{e}-Faber approximation. Similar concepts proved to be effective in the study of Hermite-Pad\'e approximation and orthogonal Hermite-Pad\'e approximation, see \cite{CacoqYsernLopezIncom,BosuwanLopez}.

\begin{definition}\label{incomsimu}\textup{ Let $F\in \mathcal{H}(E).$  Fix $m \geq m^{*}\geq 1$ and $n \in \mathbb{N}.$  Then, there exist polynomials $Q_{n,m,m^{*}}$ and $P_{n,m,m^{*},k},$ $k=0,1,\ldots,m^{*}-1,$ such that
$$ \deg(P_{n,m,m^{*},k})\leq n-1, \quad \quad \deg(Q_{n,m,m^{*}})\leq m, \quad\quad  Q_{n,m,m^{*}}\not\equiv 0.$$
$$[Q_{n,m,m^{*}} z^{k} F-P_{n,m,m^{*},k}]_j=0, \quad \quad j=0,1,\ldots,n.$$
The rational function
$R_{n,m,m^{*}}:=P_{n,m,m^{*},0}/Q_{n,m,m^{*}}$
is called  an \emph{$(n,m,m^{*})$ incomplete Pad\'{e}-Faber approximant of $F$}.}
\end{definition}

Clearly, $$[z^{k}  Q_{n,m,m^{*}}  F]_n=0, \quad \quad k=0,1,\ldots,m^{*}-1$$
and $Q_{n,m,m^{*}}$ may not be unique. For each $m\geq  m^{*}\geq 1$ and $n \in \mathbb{N}$, we choose one candidate of $Q_{n,m,m^{*}}$. Since $Q_{n,m,m^{*}}\not\equiv 0,$ we normalize it to have leading coefficient equal to $1.$ We call $Q_{n,m,m^{*}}$ \emph{a denominator of an $(n,m,m^{*})$ incomplete Pad\'{e}-Faber approximant of $F$}. Notice that for each $\alpha=1,\ldots,d,$ $Q_{n,\textup{\textbf{m}}}$  (from \eqref{denosim}) is a denominator of an $(n,|\textup{\textbf{m}}|,m_{\alpha})$ incomplete Pad\'{e}-Faber approximant of $F_{\alpha}.$

\medskip
 
Let $D_{\rho_{m^*}(F)}$ be the largest canonical region in which $F$ can be extended as a meromorphic function having at most $m^*$ poles and $\rho_{m^*}(F)$ be the index of this region.

\medskip

\begin{lemma}\label{lemma4} Let $F\in \mathcal{H}(E).$  Fix $m\geq m^{*}\geq 1.$ Suppose that there exists a polynomial $Q_{m}$ of degree $m$ such that
\begin{equation}\label{3.3332}
 \lim_{n\to \infty} Q_{n,m,m^{*}}=Q_{m}.
\end{equation}
Then,   $\rho_{0}(Q_{m} F)\geq \rho_{m^{*}}(F).$
\end{lemma}
\begin{proof}  Let  $q_{n, m,m^*}$ be the polynomial $Q_{n, m,m^*}$ normalized so that
\begin{equation}\label{contradict1}
q_{n, m,m^*}(z)=\sum_{k=0}^{m} \lambda_{n,k} z^{k}, \qquad \sum_{k=0}^{m} |\lambda_{n,k}|=1.
\end{equation}
Let $\xi$ be a pole of order $\tau$  of $ F$ in $D_{\rho_{m^*}(F)}$. Modifying conveniently the proof of \eqref{3.31}, one can show that
\begin{equation}\label{3.31^*}
\limsup_{n \rightarrow \infty} |q_{n, m,m^*}^{(j)}(\xi)|^{1/n}\leq \frac{|\Phi(\xi)|}{\rho_{m^{*}}(F)}, \qquad  j=0,1,\ldots,\tau-1.
\end{equation}
Since the sequence of polynomials $Q_{n,m,m^*}$ converges to $Q_m$, \eqref{3.31^*} entails that $\zeta$ is a zero of $Q_m$ of multiplicity at least $\tau$. Being this the case for each pole of $F$ in $D_{\rho_{m^*}(F)},$ the thesis readily follows.  \end{proof}

The following technical lemma, whose proof may be found in \cite[Lemma 3]{bosuwaninverse}, is used for proving Lemma \ref{mainlemma1}.
\begin{lemma}\label{trick}
 If a sequence of complex numbers $\{A_N\}_{N \in \mathbb{N}}$ has the following properties:
\begin{enumerate}
\item[$(i)$] $\lim_{N \rightarrow \infty} |A_N|^{1/N}=0;$
\item[$(ii)$] there exists $N_0\in \mathbb{N}$ and $C>0$ such that $|A_N|\leq C \sum_{k=N+1}^{\infty} |A_k|,$ for all $N \geq N_0$,
\end{enumerate}
then there exists $N_1\in \mathbb{N}$ such that $A_N=0$ for all $N \geq N_1.$
\end{lemma}

Lemma \ref{mainlemma1} below is the cornerstone for the proof of  Theorem \ref{inverse}.

\begin{lemma}\label{mainlemma1} Let $F\in \mathcal{H}(E).$  Fix $m\geq m^{*}\geq 1.$  Suppose that $F$ is not a rational function with at most $m^{*}-1$ poles and there exists a polynomial $Q_{m}$ of degree $m$ such that
\begin{equation}\label{3.33321}
\limsup_{n \rightarrow \infty} \|Q_{n,m,m^{*}}-Q_{m}\|^{1/n}=\theta<1.
\end{equation}
Then, the poles of $F$ in $D_{\rho_{m^{*}}(F)}$ are zeros of $Q_{m}$ counting multiplicities and, either $F$ has exactly $m^{*}$ poles in $D_{\rho_{m^{*}}(F)}$ or $\rho_0(Q_{m} F)>\rho_{m^{*}}(F).$
\end{lemma}

\begin{proof} From Lemma \ref{lemma4}, we know that the poles of $F$ in $D_{\rho_{m^{*}}(F)}$ are zeros of $Q_{m}$ counting multiplicities and $\rho_{0}(Q_{m} F)\geq \rho_{m^{*}}(F).$ Assume that $\rho_{0}(Q_{m} F)=\rho_{m^{*}}(F).$ Let us show that $F$ has exactly $m^{*}$ poles in $D_{\rho_{m^{*}}(F)}.$
To the contrary, suppose that $F$ has in $D_{\rho_{m^{*}}(F)}$ at most $m^{*}-1$ poles. Then, there exists a polynomial $q_{m^{*}}$ with $\deg q_{m^{*}}< m^{*}$ such that
$$\rho_0(q_{m^{*}} F)=\rho_{m^{*}}(F)=\rho_0(Q_{m}q_{m^{*}} F).$$ Since $\deg q_{m^{*}}< m^{*},$ by the definition of $Q_{n,m,m^{*}},$ $[Q_{n,m,m^{*}} q_{m^{*}} F]_n=0.$ Take $1< \rho< \rho_{m^{*}}(F).$ Then, by Lemma \ref{expan},
$$\frac{1}{\rho_{m^{*}}(F)}=\limsup_{n \rightarrow \infty} |[Q_{m}q_{m^{*}} F]_n|^{1/n}=\limsup_{n \rightarrow \infty} |[Q_{m}q_{m^{*}} F-Q_{n,m,m^{*}} q_{m^{*}} F]_n|^{1/n}$$
$$=\limsup_{n \rightarrow \infty} \left| \frac{1}{2\pi i}  \int_{\Gamma_{\rho}}  \frac{(Q_{m}-Q_{n,m,m^*})(z)  q_{m^{*}}(z) F(z) \Phi'(z)}{\Phi^{n+1}(z)}    dz\right|^{1/n}.$$
From the equation above, using  \eqref{3.33321}, it is easy to show that
$$\frac{1}{\rho_{m^{*}}(F)}\leq \frac{\theta}{\rho_{m^{*}}(F)},$$ which is possible only if $\rho_{m^{*}}(F) = \rho_0(q_{m^*}F)=\infty.$ Let us show that this is not so.

\medskip

From \eqref{3.33321}, without loss of generality, we can assume that $\deg Q_{n,m,m^{*}}=m.$ Set
$$q_{m^{*}}(z)F(z) =\sum_{k=0}^{\infty} a_k \Phi_k(z)$$
and
$$Q_{n,m,m^{*}}(z)=\sum_{j=0}^{m}b_{n,j}z^{j},$$
where $b_{n,m}=1.$ From  \eqref{3.33321}, we have
\begin{equation}\label{suppart}
\sup\{|b_{n,j}|: 0\leq j \leq m, \, n \in \mathbb{N}\}\leq c_{1}.
\end{equation}
Since $[Q_{n,m,m^{*}} q_{m^{*}} F]_n=0$, $[ z^{j} \Phi_k]_n = 0$ whenever $\deg(z^{j} \Phi_k) < n$ and $[ z^{m} \Phi_{n-m}]_n = \mbox{cap}^m(E)$ (see \eqref{capa}), we obtain
$$0=[Q_{n,m,m^{*}} q_{m^{*}} F]_n=\sum_{k=0}^{\infty} \sum_{j=0}^{m} a_k b_{n,j} [ z^{j} \Phi_k]_n=\sum_{k=n-m}^{\infty} \sum_{j=0}^{m} a_k b_{n,j} [ z^{j} \Phi_k]_n$$
\begin{equation}\label{almostdone}
=\textup{cap}^m(E) a_{n-m}+\sum_{k=n-m+1}^{\infty} \sum_{j=0}^{m} a_k b_{n,j} [ z^{j} \Phi_k]_n.
\end{equation}
Take $\rho > 1$. Using Lemma \ref{estimate}, for $j=0,1,\ldots,m,$ and $k\geq n-m+1$, we obtain
\begin{equation}\label{wanted2}
[ |z^{j} \Phi_k]_n|=\left|\frac{1}{2\pi i}\int_{\Gamma_{\rho}}  \frac{z^j \Phi_k(z) \Phi'(z)}{ \Phi^{n+1}(z)} dz\right| \leq c_2 \frac{\rho^k}{\rho^n}.
\end{equation}
Combining \eqref{suppart}, \eqref{almostdone}, and \eqref{wanted2}, it follows that
\begin{equation}\label{wanted1}
|a_{n-m}| \rho^{n-m}\leq c_{3} \sum_{k=n-m+1}^{\infty} |a_k| \rho^{k}.
\end{equation}
Taking $n-m = N$ and $|a_k| \rho^{k} = A_k$, \eqref{wanted1} is $(ii)$ of Lemma \ref{trick} and we also have $(i)$ because
\[\lim_{N\to \infty} |A_N|^{1/n} = \lim_{N\to \infty}(|a_N| \rho^{N})^{1/N} =  \rho/\rho_0(q_{m^*}F) = 0.\]
Consequently, there exists $N_1\in \mathbb{N}$ such that $a_N=0$ for all $N \geq N_1.$
Thus, $q_{m^{*}}F$ is a polynomial and $F$ is a rational function with at most $m^{*}-1$ poles contradicting the assumption that $F$ is not a rational function with at most $m^{*}-1$ poles. So, $F$ has exactly $m^{*}$ poles in $D_{\rho_{m^{*}}(F)}$ as we wanted to prove.
\end{proof}

\subsection{Proof of Theorem \ref{inverse}}

Before proving the main result, let us point out several important ingredients.

\medskip

Given a system of functions ${\bf F}\in \mathcal{H}(E)^d$ and a multi-index ${\bf m}\in \mathbb{N}^d,$ the space generated through polynomial combinations of the form \eqref{polycom} has dimension $\leq |{\bf m}|$. Therefore, ${\bf F}$ can have at most $|{\bf m}|$ system poles with respect to ${\bf m}$ counting multiplicities since the functions which determine the system poles and their order are of the form \eqref{polycom} and they are obviously linearly independent. For more details, see \cite[Lemma 3.5]{CacoqYsernLopez}.

\medskip

The concept of polynomial independence of a vector of functions was introduced in \cite{CacoqYsernLopez} and is also useful in this context.

\begin{definition}\label{polydef}\textup{
A vector $\textup{\textbf{F}}=(F_1,\ldots,F_d)\in \mathcal{H}(E)^d$ is said to be \emph{polynomially independent with respect to $\textup{\textbf{m}}=(m_1,\ldots,m_d)\in \mathbb{N}^d$} if there do not exist polynomials $p_1,\ldots,p_d,$ at least one of which is non-null, such that
\begin{enumerate}
\item [(i)] $\deg p_{\alpha} < m_{\alpha},$ $\alpha=1,\ldots,d,$
\item [(ii)] $\sum_{\alpha=1}^d p_{\alpha} F_{\alpha}$ is a polynomial.
\end{enumerate}}
\end{definition}

According to the assumptions of Theorem \ref{inverse}, for all $n \geq n_0,$ the polynomial $Q_{n,\textup{\textbf{m}}}$ is unique and $\deg Q_{n,\textup{\textbf{m}}}=|\textup{\textbf{m}}|.$ This implies that ${\bf F}$ is polynomially independent with respect to ${\bf m}$ for, otherwise, it is easy to see that for all sufficiently large $n$ we can construct  $(n,\textup{\textbf{m}})$ simultaneous Pad\'{e}-Faber approximants of $\textup{\textbf{F}}$ with $\deg Q_{n,\textup{\textbf{m}}}< |\textup{\textbf{m}}|$, see \cite[Lemma 3.2]{CacoqYsernLopez}. Notice that if $\textup{\textbf{F}}$ is  polynomially independent with respect to ${\bf m}$, then for each $\alpha=1,\ldots,d,$ $F_{\alpha}$ is not a rational function with at most $m_{\alpha}-1$ poles. As we pointed out in Section \ref{subsection3.1}, for each $\alpha=1,\ldots,d,$ $Q_{n,\textup{\textbf{m}}}$  is a denominator of an $(n,|\textup{\textbf{m}}|,m_{\alpha})$ incomplete Pad\'{e}-Faber approximant of $F_{\alpha}.$
 Consequently, the assumptions of Theorem \ref{inverse} allow us to make use of Lemma \ref{mainlemma1} in its proof.

\medskip

Finally, one can reduce the proof of Theorem \ref{inverse} to the case when the multi-index ${\bf m}$ has all its components equal to $1$. Indeed, given ${\bf F}\in \mathcal{H}(E)^d$ and ${\bf m}\in \mathbb{N}^d,$ define
\begin{equation}\label{definefunction}
\overline{\textup{\textbf{F}}}:=(F_1,\ldots,z^{m_1-1}F_1, F_2,\ldots, z^{m_d-1}F_d)=(f_1,f_2,\ldots,f_{|\textup{\textbf{m}}|})
\end{equation}
and
\begin{equation}\label{definemulti}
\overline{\textup{\textbf{m}}}:=(1,1,\ldots,1)
\end{equation} with $|\overline{\textup{\textbf{m}}}|=|\textup{\textbf{m}}|.$ The following assertions are easy to verify:
\begin{enumerate}
\item [(i)] the systems of equations that define $Q_{n,\textup{\textbf{m}}}$ for $\textup{\textbf{F}}$ and  $\textup{\textbf{m}}$, and $Q_{n,\overline{\textup{\textbf{m}}}}$ for $\overline{\textup{\textbf{F}}}$ and $\overline{\textup{\textbf{m}}}$ are the same.
\item [(ii)] $\textup{\textbf{F}}$ is polynomially independent with respect to $\textup{\textbf{m}}$ if and only if $\overline{\textup{\textbf{F}}}$ is polynomially independent with respect to $\overline{\textup{\textbf{m}}}.$
\item [(iii)] the poles and system poles of $(\textup{\textbf{F}},{\bf m})$ and $(\overline{\textup{\textbf{F}}}, \overline{\bf m})$, as well as their orders, coincide.
\item [(iv)]  $\rho_{m}(\textup{\textbf{F}})=\rho_{m}(\overline{\textup{\textbf{F}}}),$ for all $m \in \mathbb{N}\cup\{0\}.$
\end{enumerate}

\begin{proof}[Proof of Theorem \ref{inverse}] As shown above, without loss of generality, we can restrict our attention to the analysis of $(\overline{\bf F},\overline{\bf m})$ defined in \eqref{definefunction} and \eqref{definemulti}. Notice that \eqref{polycom} reduces to taking linear combinations of the components of $\overline{\bf F}$.  We also have that $Q_{n,\textup{\textbf{m}}}=Q_{n,\overline{\textup{\textbf{m}}}}$ and   $\overline{\textup{\textbf{F}}}$ is polynomially independent with respect to $\overline{\textup{\textbf{m}}}$.

\medskip

The arguments used in the prove follow closely those employed in proving the inverse part of \cite[Theorem 1.4]{CacoqYsernLopez};

\medskip

Choose $\beta=1,\ldots,|\textup{\textbf{m}}|$. From Lemma \ref{mainlemma1}, either $D_{\rho_1(f_{\beta})}$ contains exactly one pole of $f_{\beta}$ and it is a zero of $Q_{|\textup{\textbf{m}}|},$ or $\rho_0(Q_{|\textup{\textbf{m}}|} f_{\beta})>\rho_{1}(f_{\beta}).$ Hence, $D_{\rho_{0}(\overline{\textup{\textbf{F}}})}\not=\mathbb{C}$ and the zeros of $Q_{|\textup{\textbf{m}}|}$ contain all the poles of $f_{\beta}$ on the boundary of $D_{\rho_{0}(f_{\beta})}$ counting their order. Moreover, the function $f_{\beta}$ cannot have on the boundary of $D_{\rho_{0}(f_{\beta})}$ singularities other than poles. Thus, the poles of $\overline{\textup{\textbf{F}}}$ on the boundary of $D_{\rho_0(\overline{\textup{\textbf{F}}})}$ are zeros of $Q_{|\textup{\textbf{m}}|}$ counting multiplicities and the boundary contains no other singularity but poles. Let us call them candidate system poles of $\overline{\textup{\textbf{F}}}$ and denote them by $a_1,\ldots,a_{n_1}$ taking account of their order. They constitute a first layer of candidate system poles of $\overline{\textup{\textbf{F}}}.$

\medskip

Since $\deg Q_{|\textup{\textbf{m}}|}=|\textup{\textbf{m}}|,$ $n_1\leq |\textup{\textbf{m}}|.$ If $n_1=|\textup{\textbf{m}}|,$ we are done finding candidate system poles. Let us assume that $n_1< |\textup{\textbf{m}}|$ and let us find coefficients $c_1,\ldots,c_{|\textup{\textbf{m}}|}$ such that
$\sum_{\beta=1}^{|\textup{\textbf{m}}|} c_{\beta} f_{\beta}$
is holomorphic in a neighborhood of $\overline{D}_{\rho_0(\overline{\textup{\textbf{F}}})}.$ For this purpose we solve a homogeneous system of $n_1$ linear equations with $|\textup{\textbf{m}}|$ unknowns. In fact, if $z=a$ is a candidate system pole of $\overline{\textup{\textbf{F}}}$ with multiplicity $\tau,$ we obtain $\tau$ equations choosing the coefficients $c_\beta$ so that
\begin{equation}\label{eq21}
\int_{|w-a|=\delta} (w-a)^{k} \left( \sum_{\beta=1}^{|\textup{\textbf{m}}|} c_{\beta} f_{\beta}(w)\right)dw=0, \quad \quad k=0,\ldots,\tau-1.
\end{equation}
We write the equations for each distinct candidate system pole on the boundary of $D_{\rho_0(\overline{\textup{\textbf{F}}})}.$ This homogeneous system of linear equations has at least $|\textup{\textbf{m}}|-n_1$ linearly independent solutions, which we denote by $\textup{\textbf{c}}_{j}^{1},$ $j=1,\ldots, |\textup{\textbf{m}}|-n_1^{*},$ where $n_1^{*}\leq n_1$ denotes the rank of the system of equations.

\medskip

Let
$$\textup{\textbf{c}}_{j}^{1}:=(c_{j,1}^1,\ldots, c_{j,|{\textup{\textbf{m}}}|}^{1}), \quad \quad j=1,\ldots,  |{\textup{\textbf{m}}}|-n_1^{*}.$$ Define the $(|{\textup{\textbf{m}}}|-n_1^{*})\times |{\textup{\textbf{m}}}|$ dimensional matrix
$$C^1:= \begin{pmatrix}
 \textup{\textbf{c}}_{1}^{1}  \\
  \vdots   \\
\textup{\textbf{c}}_{|{\textup{\textbf{m}}}|-n_1^{*}}^{1}
 \end{pmatrix}.$$
 Define the vector  $\textup{\textbf{g}}_1$ of  $|{\textup{\textbf{m}}}|-n_1^{*}$ functions given by
 $$\textup{\textbf{g}}_1^{t}:=C^1 \overline{\textup{\textbf{F}}}^t=(g_{1,1},\ldots,g_{1, |{\textup{\textbf{m}}}|-n_1^{*}})^{t},$$
 where $A^t$ denotes the transpose of the matrix $A.$ Since all the rows of $C^1$ are non-null and $\overline{\textup{\textbf{F}}}$ is polynomially independent with respect to $\overline{\textup{\textbf{m}}},$ none of the functions
 $g_{1,j}=\sum_{\beta=1}^{|{\textup{\textbf{m}}}|} c_{j,\beta}^{1} f_\beta, j=1,\ldots,|{\textup{\textbf{m}}}|-n_1^{*},$
 are  polynomials.

 \medskip

 Consider the canonical domain
 $$D_{\rho_0(\textup{\textbf{g}}_1)}=\bigcap_{j=1}^{|{\textup{\textbf{m}}}|-n_1^{*}} D_{\rho_0(g_{1,j})}.$$ Clearly, $ D_{\rho_0(\overline{\textup{\textbf{F}}})}\subsetneq D_{\rho_0(\textup{\textbf{g}}_1)}$ and $[Q_{n,\overline{\textup{\textbf{m}}}}g_{1,j}]_n=0$ for all  $j=1,\ldots,|{\textup{\textbf{m}}}|-n_1^{*}.$ Therefore, for each $j=1,\ldots,|\textup{\textbf{m}}|-n_1^{*},$ $Q_{n,\overline{\textup{\textbf{m}}}}$ is a denominator of an $(n,|\overline{\textup{\textbf{m}}}|,1)$ incomplete Pad\'{e}-Faber approximant of $g_{1,j}$. Since the $g_{1,j}$ are not polynomials, by Lemma \ref{mainlemma1} with $m^{*}=1,$ for each $j=1,\ldots,|{\textup{\textbf{m}}}|-n_1^{*},$ either $D_{\rho_1(g_{1,j})}$ contains exactly one pole of $g_{1,j}$ and it is a zero of $Q_{|\textup{\textbf{m}}|},$ or $\rho_{0}(Q_{|\textup{\textbf{m}}|}g_{1,j})>\rho_{1}(g_{1,j}).$ In particular, $D_{\rho_0(\textup{\textbf{g}}_1)}\not=\mathbb{C}$ and all the singularities of $\textup{\textbf{g}}_1$ on the boundary of $D_{\rho_0(\textup{\textbf{g}}_1)}$ are poles which are zeros of $Q_{|\textup{\textbf{m}}|}$ counting their order. They form the next layer of candidate system poles of $\overline{\textup{\textbf{F}}}.$

 \medskip

Denote by $a_{n_1+1},\ldots,a_{n_1+n_2}$ these new candidate system poles. Again, if $n_1+n_2=|\textup{\textbf{m}}|,$ we are done. Otherwise, $n_2<|\textup{\textbf{m}}|-n_1\leq |\textup{\textbf{m}}|-n_{1}^{*},$ and we repeat the same process eliminating the $n_2$ poles $a_{n_1+1},\ldots,a_{n_1+n_2}.$ We have $|\textup{\textbf{m}}|-n_{1}^{*}$ functions which are holomorphic on $D_{\rho_0(\textup{\textbf{g}}_1)}$ and meromorphic in a neighborhood of $\overline{D}_{\rho_{0}(\textup{\textbf{g}}_1)}.$ The corresponding homogeneous system of linear equations, similar to \eqref{eq21}, has at least $|\textup{\textbf{m}}|-n_1^{*}-n_2$ linearly independent solutions $\textup{\textbf{c}}_j^2,$ $j=1,\ldots,|\textup{\textbf{m}}|-n_1^{*}-n_2^{*},$ where $n_2^{*}\leq n_2$ is the rank of the new system. Let
$$\textup{\textbf{c}}_{j}^{2}:=(c_{j,1}^2,\ldots, c_{j,|{\textup{\textbf{m}}}|}^{2}), \quad \quad j=1,\ldots,  |{\textup{\textbf{m}}}|-n_1^{*}-n_2^{*}.$$ Define the $(|{\textup{\textbf{m}}}|-n_1^{*}-n_2^{*})\times (|{\textup{\textbf{m}}}|-n_1^{*})$ dimensional matrix
$$C^2:= \begin{pmatrix}
 \textup{\textbf{c}}_{1}^{2}  \\
  \vdots   \\
\textup{\textbf{c}}_{|{\textup{\textbf{m}}}|-n_1^{*}-n_2^{*}}^{2}
 \end{pmatrix}.$$
 Define the vector  $\textup{\textbf{g}}_2$ with  $|{\textup{\textbf{m}}}|-n_1^{*}-n_2^{*}$ functions given by
 $$\textup{\textbf{g}}_2^{t}:=C^2\textup{\textbf{g}}_1^t=C^2C^1 \overline{\textup{\textbf{F}}}^t=(g_{2,1},\ldots,g_{2, |{\textup{\textbf{m}}}|-n_1^{*}-n_2^{*}})^{t}.$$
As $C^1$  and $C^2$ have full rank, so does  $C^2 C^1$. Therefore, the rows of $C^2 C^1$ are linearly independent; in particular, they are non-null. Thus, all component functions of $\textup{\textbf{g}}_2$ are not polynomials, due to the polynomial independence of $\overline{\textup{\textbf{F}}}$ with respect to $\overline{\textup{\textbf{m}}},$ and we can apply again  Lemma \ref{mainlemma1}. The proof is completed using finite induction.

\medskip

On each layer of system poles, $1 \leq n_k\leq |\textup{\textbf{m}}|.$ Therefore, in a finite number of steps, say $N-1,$ their sum equals $|\textup{\textbf{m}}|.$  Consequently, the number of candidate system poles of $\overline{\textup{\textbf{F}}}$ in some canonical region, counting multiplicities, is exactly equal to $|\textup{\textbf{m}}|$ and they are precisely the zeros of $Q_{|\textup{\textbf{m}}|}$ as we wanted to prove.

\medskip

Summarizing, in the $N-1$ steps, we have produced $N$ layers of candidate system poles. Each layer contains $n_k$ candidates, $k=1,\ldots,N.$ At the same time, on each step $k,$ $k=1,\ldots,N-1,$ we have solved a linear system of $n_k$ equations, of rank $n_k^{*},$ with $|\textup{\textbf{m}}|-n_1^{*}-\cdots-n_k^{*},$ $n_k^{*}\leq n_k,$ linearly independent solutions. We find ourselves on the $N$-th layer which has $n_N$ candidate system poles.

\medskip

Let us try to eliminate the poles on the last layer. For that purpose, define the corresponding homogeneous system of linear equations as in \eqref{eq21}, and we get
$$n_N=|\textup{\textbf{m}}|-n_1-\cdots-n_{N-1}\leq |\textup{\textbf{m}}|-n_1^{*}-\cdots-n_{N-1}^{*}=:\overline{n}_N$$
equations with $\overline{n}_N$ unknowns. For each candidate system pole $a$ of multiplicity $\tau$ on the $N$-th layer, we impose the equations
\begin{equation}\label{eq22}
\int_{|w-a|=\delta} (w-a)^k\left(\sum_{\beta=1}^{\overline{n}_N} c_{\beta} g_{N-1,\beta}(w) \right)dw=0,\quad \quad k=0,\ldots,\tau-1,
\end{equation}
where $\delta$ is sufficiently small and the $g_{N-1,\beta},$ $\beta=1,\ldots,\overline{n}_N,$ are the functions associated with the linearly independent solutions produced on step $N-1.$

\medskip

Let $n_N^{*}$ be the rank of this last homogeneous system of linear equations. Assume that $n_k^{*}<n_k$ for some $k\in \{1,\ldots,N\}.$ Then, the rank of the last system of equations is strictly less than the number of unknowns, namely $n_N^{*}<\overline{n}_N.$ Repeating the same process, there exists a vector of functions $$\textup{\textbf{g}}_N:=(g_{N,1},\ldots, g_{N,|\textup{\textbf{m}}|-n_1^{*}-\cdots-n_{N}^{*}})$$ such that none of $g_{N,\beta}$ is a polynomial because of the polynomial independence of $\overline{\textup{\textbf{F}}}$ with respect to $\overline{\textup{\textbf{m}}}.$ Applying Lemma \ref{mainlemma1}, each $g_{N,\beta}$ has on the boundary of its canonical domain of analyticity a pole which is a zero of $Q_{|\textup{\textbf{m}}|}.$ However, this is impossible because all the zeros of $Q_{|\textup{\textbf{m}}|}$ are strictly contained in a smaller domain. Consequently, $n_k=n_k^{*},$ $k=1,\ldots,N.$

\medskip

We conclude that all the $N$ homogeneous systems of linear equations that we have solved have full rank. This implies that if in any one of those $N$ systems of equations we equate one equation to $1$ instead of zero (see \eqref{eq21} or \eqref{eq22}), the corresponding nonhomogeneous system of linear equations has a solution. By the definition of a system pole, this implies that each candidate system pole is indeed a system pole of order at least equal to its multiplicity as zero of $Q_{|\textup{\textbf{m}}|}.$ However,   $\overline{\textup{\textbf{F}}}$ can have at most $|\textup{\textbf{m}}|$ system poles with respect to $\overline{\textup{\textbf{m}}}$; therefore, all candidate system poles are system poles, and their order coincides with the multiplicity of that point as a zero of $Q_{|\textup{\textbf{m}}|}.$ This also means that $Q_{|\textup{\textbf{m}}|}=Q_{\textup{\textbf{m}}}^{\textup{\textbf{F}}}.$ We have completed the proof.
\end{proof}

\section{Acknowledgements}

\quad \quad  The authors would like to thank Dr. Olivier S\`ete for his suggestion for references of Lemma \ref{estimate}.

\noindent Nattapong Bosuwan\\
         Department of Mathematics,
         Mahidol University\\
         Rama VI Road, Ratchathewi District,\\
         Bangkok 10400, Thailand\\
and \\
 \noindent Centre of Excellence in Mathematics, CHE,\\
                 Si Ayutthaya Road,\\
                 Bangkok 10400, Thailand\\
\noindent email: nattapong.bos@mahidol.ac.th,\\

\noindent\\
Guillermo L\'opez Lagomasino \\
Mathematics Department,\\
Universidad Carlos III de Madrid, \\
c/ Universidad, 30\\
28911, Legan\'es, Spain \\
email: lago@math.uc3m.es

\end{document}